\documentclass[11pt,a4paper]{amsart}

\usepackage[T1]{fontenc}
\usepackage[utf8]{inputenc}
\usepackage{lmodern}
\usepackage{microtype}
\usepackage[a4paper,margin=1in]{geometry}
\usepackage{amsmath,amssymb,mathtools}
\usepackage{enumitem}
\usepackage{xcolor}
\usepackage{graphicx}
\usepackage{float}
\usepackage{hyperref}
\usepackage[nameinlink,capitalize,noabbrev]{cleveref}

\hypersetup{
  colorlinks=true,
  linkcolor=blue!45!black,
  citecolor=blue!45!black,
  urlcolor=blue!45!black,
  pdftitle={Quadratic Scalar Curvature Decay and Uniform Positivity in Dimensions Four Through Seven},
  pdfauthor={Zhehui Wang and Jintian Zhu}
}

\setlist[enumerate,1]{
  label=(\arabic*),
  leftmargin=2.2em,
  itemsep=0.25em,
  topsep=0.35em
}
\setlist[itemize]{
  leftmargin=2em,
  itemsep=0.2em,
  topsep=0.35em
}

\newtheorem{theorem}{Theorem}[section]
\newtheorem{conjecture}[theorem]{Conjecture}
\newtheorem{corollary}[theorem]{Corollary}
\newtheorem{proposition}[theorem]{Proposition}
\newtheorem{lemma}[theorem]{Lemma}
\newtheorem{definition}[theorem]{Definition}
\newtheorem{remark}[theorem]{Remark}
\newtheorem{question}[theorem]{Question}

\crefname{conjecture}{conjecture}{conjectures}
\Crefname{conjecture}{Conjecture}{Conjectures}
\crefname{question}{question}{questions}
\Crefname{question}{Question}{Questions}

\newcommand{\diver}{\operatorname{div}}
\newcommand{\sign}{\operatorname{sign}}
\newcommand{\dist}{\operatorname{dist}}
\newcommand{\spt}{\operatorname{spt}}
\newcommand{\Int}{\operatorname{Int}}

\title[Quadratic scalar curvature decay]
{Quadratic Scalar Curvature Decay and Uniform Positivity
in Dimensions Four through Seven}
\author{Zhehui Wang}
\address{School of Sciences, Great Bay University, Dongguan 523000, China}
\email{wangzhehui@gbu.edu.cn}
\author{Jintian Zhu}
\address{Institute for Theoretical Sciences, Westlake University, Hangzhou 310030, China}
\email{zhujintian@westlake.edu.cn}
\date{}

\begin{document}

\makeatletter
\tagsleft@false
\let\veqno\@@eqno
\makeatother

\begin{abstract}
Let $4\le n\le7$ and let $(M^n,g)$ be a complete, connected,
orientable, noncompact Riemannian manifold of positive scalar curvature.
We prove that if the asymptotic quadratic scalar curvature coefficient of
$g$ is greater than $(n-1)/n$, then $M$ carries a complete smooth
metric whose scalar curvature is at least one.  The threshold $(n-1)/n$
and the strict inequality are optimal. This confirms the second part of Gromov's critical rate of decay conjecture in dimensions four through seven. 
\end{abstract}

\maketitle

\section{Introduction}

On a closed manifold, a complete metric of positive scalar curvature automatically has
scalar curvature bounded away from zero.  On an open manifold, however, positivity
and uniform positivity are genuinely different: a complete metric may have positive scalar curvature while the infimum of the scalar curvature is zero. This leads to the following question.

\begin{question}\label{ques:intro-decay}
How slowly can the scalar curvature of a complete positive-scalar-curvature
metric decay before the underlying manifold must admit a complete metric
with uniformly positive scalar curvature?
\end{question}

Let $(M^n,g)$ be a complete, connected, orientable, noncompact Riemannian manifold
with positive scalar curvature.  Fix a base point $p\in M$ and write
$$
  r_p(x)=d_g(p,x).
$$
Define the
\emph{asymptotic quadratic scalar curvature coefficient} by
\begin{equation}\label{eq:intro-coefficient}
  \mathcal C_{\infty,p}(g)
  :=\liminf_{r_p(x)\to\infty}r_p(x)^2R_g(x).
\end{equation}
The coefficient is independent of the choice of $p$, so we omit $p$
from the notation.

For every threshold $T\ge0$, the condition
$\mathcal C_\infty(g)>T$ is equivalent to the existence of a constant
$C>T$ and a compact set $K\Subset M$ such that
\begin{equation}\label{eq:intro-decay}
  R_g(x)>\frac{C}{r_p(x)^2},
  \qquad \forall x\in M\setminus K.
\end{equation}

We recall Gromov's critical rate of decay conjecture
\cite[Section~3.6.1]{Gromov2023}, which is reformulated as follows.

\begin{conjecture}[Gromov]
\label{conj:gromov-critical}
For every $n\ge3$, there exists a dimensional constant $C_n>0$ with the
following property.  Let $M$ be a connected orientable noncompact
$n$-manifold that admits a complete metric of positive scalar curvature.
\begin{enumerate}
\item[(G1)] For every $0<C<C_n$, the manifold $M$ admits a complete metric
  $g$ of positive scalar curvature such that
  $$
    \mathcal C_\infty
(g)>C.  $$
\item[(G2)] If $M$ admits a complete metric $g$ of positive scalar curvature
  satisfying
  $
    \mathcal C_\infty(g)>C_n
  $,
  then $M$ admits a complete metric $\bar g$ satisfying
  $$
    \inf_MR_{\bar g}>0.
  $$
\end{enumerate}
\end{conjecture}

Chen proved (G2) in dimension three with the sharp threshold
$C_3=2/3$, and also conjectured that the sharp threshold is $$C_n=\frac{n-1}{n}$$ in general
\cite[Corollary~1.6 and the discussion following Proposition~1.7]{Chen2026}.
Our main theorem verifies this in dimensions
$4\le n\le7$, while (G1) remains open.

\begin{theorem}
\label{thm:intro-main}
Let $4\le n\le7$, and let $(M^n,g)$ be a complete, connected, orientable,
noncompact Riemannian manifold with $R_g>0$.  Suppose that
$$
  \mathcal C_\infty(g)>\frac{n-1}{n}.
$$
Then $M$ admits a complete smooth metric $\bar g$ satisfying
$$
  R_{\bar g}\ge1.
$$
\end{theorem}

\begin{remark}
Chen's warped-product example on $\mathbb R^2\times T^{n-2}$ \cite[Section~2]{Chen2026}, together with
the obstruction in
\cite[p.~648]{Gromov2018}, shows that the coefficient $(n-1)/n$ and the
strict inequality in the hypothesis of \cref{thm:intro-main} are both
optimal.
\end{remark}

\subsection*{Background and related work}

Gromov formulated the conjecture in his study of scalar curvature bounds on
open manifolds \cite[Section~3.6.1]{Gromov2023}.  We review below the results
directly concerning quadratic scalar curvature decay and the passage from
positive to uniformly positive scalar curvature.

The three-dimensional background begins with the topological classification
of manifolds with uniformly positive scalar curvature.
Chang--Weinberger--Yu \cite[Theorem~3]{ChangWeinbergerYu2010} proved that
a connected orientable $3$-manifold with finitely generated fundamental
group admitting a complete metric of uniformly positive scalar curvature
is homeomorphic to a possibly infinite connected sum of spherical
$3$-manifolds and copies of $S^2\times S^1$.
Bessi\`eres--Besson--Maillot \cite[Theorem~1.1]{BessieresBessonMaillot2011}
proved the corresponding decomposition for complete metrics of bounded
geometry and uniformly positive scalar curvature using Ricci flow; in this
case, only finitely many diffeomorphism types of spherical summands occur. Later,
Gromov \cite[Section~3.10.2]{Gromov2023} and Wang
\cite[Theorem~1.1]{Wang2023} independently established the decomposition
for every complete connected orientable $3$-manifold of uniformly positive
scalar curvature, without any finite-generation or bounded-geometry
assumption. Their arguments use $\mu$-bubbles to construct compact exhaustions
with spherical boundary; Wang also uses minimal surfaces to identify the
prime factors.

The subsequent decay results retain this topological conclusion while
weakening the curvature hypothesis. Balacheff, Gil Moreno de Mora
Sard\`a, and Sabourau
\cite[Theorem~1.3 and Corollary~1.5]{BalacheffGilMorenoSabourau2025}
proved that a complete connected orientable Riemannian $3$-manifold with
$R_g>0$ and $\mathcal C_\infty(g)>64\pi^2$ has the same connected-sum
decomposition and consequently admits a complete metric of uniformly
positive scalar curvature. Their proof uses fill-radius estimates.
Chen \cite[Theorem~1.5 and Corollary~1.6]{Chen2026} used $\mu$-bubbles to
prove both conclusions under the sharper hypothesis
$\mathcal C_\infty(g)>2/3$, and showed that the threshold $2/3$ is optimal.

Chen's argument also identifies the higher-dimensional coefficient relevant
here.  For $3\le n\le7$, the hypothesis
$$
  \mathcal C_\infty(g)>\frac{n-1}{n}
$$
produces a compact exhaustion whose boundaries are of positive Yamabe type
\cite[Proposition~1.8]{Chen2026}.  This construction uses the
generalized soap bubble method; compare
Chodosh--Li \cite{ChodoshLi2024}.  Recent work on complete uniformly positive
scalar curvature in higher dimensions includes topological obstruction
results of Chodosh--M\'aximo--Mukherjee \cite{ChodoshMaximoMukherjee2024} and Sweeney \cite{Sweeney2026} and construction criteria
of Das \cite{Das2026}.

The conclusion that each interface is of positive Yamabe type concerns only
its intrinsic geometry.  An effective way to prove (G2) of
Gromov's conjecture is to solve the block filling problem: on every
compact cobordism between consecutive interfaces, construct a positive-scalar-curvature
metric that induces the chosen boundary metrics and is a product near the
boundary.  Relative Yamabe theory and conformal cobordism provide the
relevant extension theorem: Akutagawa and Botvinnik \cite{AkutagawaBotvinnik2002a,AkutagawaBotvinnik2002b} showed that positivity of the relative Yamabe invariant with a prescribed positive-scalar-curvature
boundary metric yields an extension that is an exact product near the
boundary.  Related
analytic results on the conformal Laplacian and positive-scalar-curvature
metrics on manifolds with boundary are given by
Rosenberg--Ruberman--Xu \cite{RosenbergRubermanXu2026}.

\subsection*{Outline of the proof}

The proof consists of three steps, illustrated in
\cref{fig:construction-overview}.  First, Chen's globally minimizing
$\mu$-bubbles yield a compact exhaustion $K_i$ and smooth interfaces
$\Sigma_i=\partial K_i$, decomposing $M$ into compact blocks
$$
 W_0=K_1,\qquad
W_i=K_{i+1}\setminus\operatorname{Int}K_i\quad(i\ge1).
$$
The decay hypothesis then yields the positive function $q_i$ defined in \eqref{eq:intro-q}, and stability provides
positive-scalar-curvature metrics $k_i$ on the interfaces.  Second, one-sided calibration fields obtained from the global
minimizing property make the conformal Robin quadratic form on
each $W_i$ coercive.  Relative Yamabe filling
then produces a positive-scalar-curvature metric $G_i$ satisfying
\(G_i=dt^2+k_j\) near each boundary interface \(\Sigma_j\).  Third, the block metrics are rescaled
independently and consecutive product collars are joined by sufficiently
long warped cylinders.  Choosing the neck lengths appropriately gives
$R_{\bar g}\ge1$ and completeness.

\begin{figure}[H]
\centering
\includegraphics[width=\textwidth]{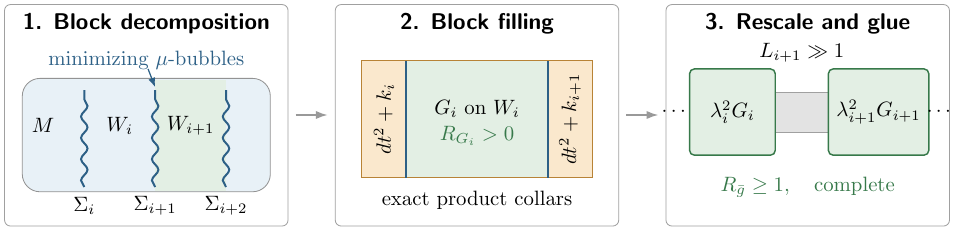}
\caption{The three stages of the construction: block decomposition by
minimizing $\mu$-bubbles, positive-scalar-curvature filling with exact
product collars, and global gluing by rescaling and long warped necks.  The
shapes are schematic.}
\label{fig:construction-overview}
\end{figure}

\subsection*{The block filling argument}

We now explain the block filling argument in detail.
Let $\Omega_i\subset A_i$ be one of Chen's globally minimizing
regions, let $\Sigma_i=\partial\Omega_i\cap A_i^\circ$ be its
$\mu$-bubble boundary, and let $h_i$ be the prescribed function.  The function $h_i$ constructed by Chen satisfies the pointwise
inequality
\begin{equation}\label{eq:intro-q}
  q_i=R_g+\frac{n}{n-1}h_i^2-2|\nabla h_i|>0.
\end{equation}
Stability yields only an intrinsic positive-scalar-curvature metric on
$\Sigma_i$.

To carry out block filling, we establish positivity of the relative
Yamabe invariant of each compact block, with the boundary conformal
class induced by $g$. Since $R_g>0$, this positivity would follow
immediately if the boundary of the block were mean convex with respect
to its outward normal. Indeed, the conformal Robin quadratic form
$$
E_g(u)
=
\frac{4(n-1)}{n-2}\int_W|\nabla u|^2\,dV_g
+\int_W R_gu^2\,dV_g
+2\int_{\partial W}H_gu^2\,dA_g
$$
would then be coercive: the boundary term is nonnegative, while
the interior terms control the $H^1(W)$-norm because $W$ is compact
and $R_g>0$.

For the blocks bounded by $\mu$-bubbles, however, this favorable
sign condition is not automatic.
We handle the negative boundary contributions by a calibration
argument using the global minimizing property of the $\mu$-bubble.
The calibration fields allow us to rewrite $E_g$ using a nonnegative
gradient term involving $|\nabla u+Xu|^2$, for a suitable vector field $X$.
In this identity, the remaining interior coefficient is strictly positive
and the boundary coefficients are nonnegative. This proves coercivity on
each block and hence the required relative Yamabe positivity.

More precisely, fix a compact block $W$ and orient the prescribed
functions $f_\alpha$ along its boundary components $\Sigma_\alpha$
by the outward normal $\nu_W$.
Convex relaxation of one-sided minimization, exact penalization,
and the Dirichlet total variation subgradient theorem produce fields
$Z_\alpha$ with bounded divergence on disjoint compact neighborhoods
$U_\alpha$, satisfying
$$
 |Z_\alpha|\le1,\qquad
 \operatorname{div}Z_\alpha\le f_\alpha
 \quad\text{in }U_\alpha,\qquad
 [Z_\alpha,\nu_W]=1
 \quad\text{on }\Sigma_\alpha,
$$
where brackets denote normal traces.
Write $s^\pm=\max\{\pm s,0\}$ and put
$$
 X_\alpha=-\frac{n-2}{2(n-1)}\,f_\alpha^-Z_\alpha.
$$
Choose $U_\alpha$ to contain $\operatorname{spt}(f_\alpha^-)$,
with $f_\alpha^-=0$ near its artificial boundary.
The zero extensions of $X_\alpha$ therefore give a bounded field
$X=\sum_\alpha X_\alpha$ on $W$ with bounded divergence.
Expanding the square and integrating by parts yields
$$
\begin{aligned}
E_g(u)
={}&\frac{4(n-1)}{n-2}\int_W|\nabla u+Xu|^2\,dV_g\\
&+\int_W\bigl(R_g+\frac{4(n-1)}{n-2}
  (\diver X-|X|^2)\bigr)u^2\,dV_g\\
&+\int_{\partial W}\bigl(2H-\frac{4(n-1)}{n-2}
  [X,\nu]\bigr)u^2\,dA_g.
\end{aligned}
$$
The band inequalities $q_i>0$, together with $R_g>0$, yield
$$
 \int_W
 \bigl(R_g+a_n(\operatorname{div}X-|X|^2)\bigr)u^2\,dV_g
 \ge \delta_W\int_W u^2\,dV_g
$$
for some $\delta_W>0$ and every $u\in H^1(W)$.
On $\Sigma_\alpha$, the normal trace and $H_g=f_\alpha$ give
$$2H_g-\frac{4(n-1)}{n-2}[X,\nu_W]=2f_\alpha^+\ge0.$$
These estimates make $E_g$ coercive.
The critical Sobolev inequality then gives a positive fixed-class
relative Yamabe constant and hence a positive relative Yamabe
invariant on each block.

The principal new ingredient is this calibration argument, which converts
the global minimizing property of Chen's $\mu$-bubbles into coercive
estimates for the conformal Robin quadratic form on both sides of every
interface.
The extension theorem of Akutagawa and Botvinnik
\cite{AkutagawaBotvinnik2002a} then supplies positive-scalar-curvature metrics on the
blocks that equal $dt^2+k_i$ near each interface $\Sigma_i$.

\subsection*{Organization}

Section \ref{sec2} gives the block decomposition using Chen's $\mu$-bubble method,
including the global minimizing property and the precise identification of
the two sides of each interface.  Section \ref{sec3} carries out the block
construction: it derives the BV calibration, proves the conformal Robin and
relative Yamabe estimates, and applies the exact-boundary filling theorem.
Section \ref{sec4} performs the global construction: it rescales the block metrics,
inserts quantitative warped necks, and proves completeness.  Appendix \ref{appendixA}
recalls relative Yamabe theory and the Akutagawa--Botvinnik filling theorem.
Appendix \ref{appendixB} records the Dirichlet total variation subgradient and normal trace
results used in the calibration.

\subsection*{Acknowledgements.}
We thank both Kazuo Akutagawa and Demetre Kazaras for drawing our attention
to an issue in the proof of Lemma~6 in
\cite{AkutagawaBotvinnik2002a}. We are grateful to Kazuo Akutagawa and
Boris Botvinnik for their helpful correspondence, for communicating
their revised statement and proof of the lemma, and for permitting
us to include this material in Appendix~\ref{appendixA}.

\subsection*{Disclosure on AI assistance.} The authors used AI-assisted tools, principally ChatGPT. The authors wrote and verified all theorem statements, proofs, and they take full responsibility for the contents of the paper.
\section{Block decomposition}\label{sec2}

In this section, we use the $\mu$-bubble exhaustion established by Chen
\cite[Proposition~1.8]{Chen2026} to decompose $M$ into compact blocks
separated by smooth hypersurfaces.

For a Caccioppoli set $F$ and an open set $U$, write
$$
 P(F;U)=|D\chi_F|(U),
$$
and denote its reduced boundary by $\partial^*F$.  By a smooth band we mean
a compact smooth codimension-zero submanifold whose
boundary is partitioned into two unions of components.  For an oriented
hypersurface with unit normal $\nu$, our sign convention is
$$
 \mathrm{II}(X,Y)=\langle\nabla_X\nu,Y\rangle,
 \qquad H=\operatorname{tr}\mathrm{II}.
$$

\begin{definition}[Minimizing $\mu$-bubble]
\label{def:minimizing-mu-bubble}
Let $A$ be a compact smooth band with
$\partial A=\partial_{\rm in}A\sqcup\partial_{\rm out}A$, and let
$h\in C^\infty(A^\circ)$. Fix a smooth reference region
$\widehat\Omega\subset A$ which contains a relative collar of
$\partial_{\rm in}A$ and is disjoint from a relative collar of
$\partial_{\rm out}A$. A smooth hypersurface
$\Sigma=\partial\Omega\cap A^\circ$ is a \emph{minimizing $\mu$-bubble
for $h$ in $A$} if $\Omega$ is a smooth global minimizer of
\begin{equation}\label{eq:mu-bubble-functional}
 \mathcal A_h(F)
 =P(F;A^\circ)
  -\int_{A^\circ}(\chi_F-\chi_{\widehat\Omega})h\,dV_g
\end{equation}
among all Caccioppoli sets with
$F\Delta\widehat\Omega\Subset A^\circ$.
In particular, $\Omega$ has the same collar behavior as
$\widehat\Omega$ at both boundary parts of the band.
\end{definition}

The following proposition is extracted from
\cite[Proposition~1.8]{Chen2026}, using
\cite[Proposition~3.1 and Lemmas~3.2--3.4]{Chen2026}.

\begin{proposition}[Chen]
\label{prop:block-decomposition}
Let $4\le n\le7$, and let $(M^n,g)$ be a complete, connected, orientable,
noncompact Riemannian manifold with $R_g>0$ and
$$
  \mathcal C_\infty(g)>\frac{n-1}{n}.
$$
Given a compact set $K_0\Subset M$, there are

\begin{itemize}
\item pairwise disjoint compact smooth bands $A_i\Subset M$;
\item smooth functions $h_i\colon A_i^\circ\to\mathbb R$;
\item smooth Caccioppoli sets $\Omega_i\subset A_i$;
\item smooth embedded closed hypersurfaces
  $$
    \Sigma_i=\partial\Omega_i\cap A_i^\circ;
  $$
\item a compact smooth exhaustion $\{K_i\}_{i\ge1}$ of $M$ by
codimension-zero submanifolds, with
  $$
    K_i\Subset\operatorname{Int}K_{i+1}\quad(i\ge0),
    \qquad M=\bigcup_{i=1}^\infty K_i,
    \qquad \partial K_i=\Sigma_i,
  $$
\end{itemize}

with the following properties.

\begin{enumerate}[label=\textup{(\roman*)},ref=\roman*]
\item\label{item:band-potential}
On $A_i^\circ$,
\begin{equation}\label{eq:band-potential}
  q_i:=R_g+\frac{n}{n-1}h_i^2-2|\nabla h_i|>0.
\end{equation}

\item\label{item:band-collars}
Writing
$\partial A_i=\partial_{\rm in}A_i\sqcup
\partial_{\rm out}A_i$, one has
$$
  h_i\longrightarrow+\infty
  \quad\text{at }\partial_{\rm in}A_i,
  \qquad
  h_i\longrightarrow-\infty
  \quad\text{at }\partial_{\rm out}A_i.
$$
\item\label{item:band-minimizer}
The set $\Omega_i$ is a minimizing region for $h_i$ in $A_i$
in the sense of \cref{def:minimizing-mu-bubble}.

\item\label{item:band-variation}
If $\nu_i$ is the outward unit normal of $\Omega_i$,
then
$$
  H_{\Sigma_i}=h_i
$$
and, for every $\psi\in C^\infty(\Sigma_i)$,
\begin{equation}\label{eq:interface-stability}
  \int_{\Sigma_i}
  \left(|\nabla_{\Sigma_i}\psi|^2
       +\frac12R_{\Sigma_i}\psi^2\right)dA
  \ge
  \frac12\int_{\Sigma_i}q_i\psi^2\,dA.
\end{equation}

\item\label{item:band-normal-orientation}
For every $i$, the outward unit normal of $K_i$ along $\Sigma_i$ is
$\nu_i$.
\end{enumerate}

\noindent
The bands, hypersurfaces, and blocks need not be connected.  Statements
concerning disconnected objects are interpreted componentwise.
\end{proposition}

\begin{proof}
For the reader's convenience, we briefly sketch the construction here. Choose constants
$\frac{n-1}{n}<C_*<C<\mathcal C_\infty(g)$.  Replacing $K_0$ by a
larger compact set does not weaken the conclusion, so by
\eqref{eq:intro-decay} we may assume
$$
  R_g>\frac{C}{r_p^2}\quad\text{on }M\setminus K_0.
$$
Choose $\varepsilon>0$ sufficiently small and set
$$
  \mu=\frac12\sqrt{\frac{nC_*}{n-1}-1},
  \qquad
  f(s)=\frac{2(n-1)}{ns}
       \left(-\frac12+\mu\tan(\mu\log s)\right).
$$
This profile satisfies
$$
  2f'(s)-\frac{n}{n-1}f(s)^2=\frac{C_*}{s^2}.
$$
Its consecutive pole intervals are $(\alpha_k,\beta_k)$, where
$$
  \alpha_k=\exp\!\left(\mu^{-1}\left(k\pi-\frac\pi2\right)\right),
  \qquad
  \beta_k=\exp\!\left(\mu^{-1}\left(k\pi+\frac\pi2\right)\right).
$$
Following \cite[proof of Proposition~1.8]{Chen2026}, choose an integer $k_*$
so large that
$$
  \alpha_{k_*}>\max_{K_0}r_p+1.
$$
For $i\ge1$, set $k_i=k_*+2(i-1)$ and
$$
  a_i=(1+\varepsilon)\alpha_{k_i},
  \qquad b_i=(1+\varepsilon)\beta_{k_i}.
$$
By the Greene--Wu smoothing theorem
\cite[Proposition~2.1]{GreeneWu}, after a sufficiently small constant
shift, there is a smooth proper function $\rho$ such that
$$
  |\rho-r_p|<\varepsilon,
  \qquad |\nabla\rho|<1+\varepsilon,
$$
and all $a_i$ and $b_i$ are regular values.  The choice of $k_*$ gives
$K_0\Subset\{\rho<a_1\}$.  Define
$$
  A_i=\{a_i\le\rho\le b_i\},
  \qquad
  h_i=-f\!\left(\frac{\rho}{1+\varepsilon}\right)
  \quad\text{on }A_i^\circ.
$$
Set
$$
  \partial_{\rm in}A_i=\{\rho=a_i\},
  \qquad
  \partial_{\rm out}A_i=\{\rho=b_i\}.
$$
Then the $A_i$ are pairwise disjoint compact smooth bands,
$b_i<a_{i+1}$, and $a_i\to\infty$.  The pole limits of $f$ give
$$
  h_i\longrightarrow+\infty\quad\text{at }\partial_{\rm in}A_i,
  \qquad
  h_i\longrightarrow-\infty\quad\text{at }\partial_{\rm out}A_i.
$$
Using the scalar curvature lower bound, the differential equation satisfied
by $f$, the estimates for $\rho$, and the choice of $\varepsilon$, we obtain
$q_i>0$ on $A_i^\circ$.

Choose a regular value $\ell_i\in(a_i,b_i)$ and set
$$
  \widehat\Omega_i=\{a_i\le\rho\le\ell_i\}\subset A_i.
$$
By \cite[Proposition~3.1]{Chen2026}
(see also \cite[Proposition~2.1]{Zhu2021}), there is a global minimizer
$\Omega_i$ with
$$
  \Omega_i\Delta\widehat\Omega_i\Subset A_i^\circ.
$$
For $n\le7$,
$\Sigma_i=\partial\Omega_i\cap A_i^\circ$ is a smooth embedded hypersurface
\cite[Proposition~3.1]{Chen2026}.  We take the corresponding smooth open
representative of $\Omega_i$.
The first and second variation formulas
\cite[Lemmas~3.2--3.4]{Chen2026} give $H_{\Sigma_i}=h_i$ and the stability
inequality.  Combining the latter with the traced Gauss equation and
$$
  |\mathrm{II}_i|^2\ge\frac{h_i^2}{n-1},
  \qquad \nu_i(h_i)\ge-|\nabla h_i|,
$$
yields \eqref{eq:interface-stability}.

Finally, define
$$
  K_i=\overline{\{\rho<a_i\}\cup
  (\Omega_i\cap A_i^\circ)}.
$$
The collar behavior gives $\partial K_i=\Sigma_i$, while properness of
$\rho$, the inequalities $b_i<a_{i+1}$, and $a_i\to\infty$ give compactness,
nesting, and exhaustion.  Near $\Sigma_i$, the chosen open representative
shows that $K_i$ lies on the $\Omega_i$-side of $\Sigma_i$.  Hence its
outward normal along $\Sigma_i$ is $\nu_i$, completing the proof.
\end{proof}

\begin{proposition}
\label{prop:interface-psc}
Let $\gamma_i=g|_{\Sigma_i}$.  The first eigenvalue of the conformal
Laplacian of every connected component of $(\Sigma_i,\gamma_i)$ is positive.
Consequently, one can choose a metric
$$
  k_i\in[\gamma_i]
$$
with $R_{k_i}>0$ on every component.
\end{proposition}

\begin{proof}
Put $m=n-1\ge3$ and
$$
  L_{\gamma_i}=-\frac{4(m-1)}{m-2}\Delta_{\gamma_i}
  +R_{\gamma_i}.
$$
Fix a connected component $\Gamma\subset\Sigma_i$.  Compactness and
\eqref{eq:band-potential} give
$\delta_{i,\Gamma}:=\min_\Gamma q_i>0$.  For every nonzero
$\psi\in C^\infty(\Gamma)$, \eqref{eq:interface-stability} yields
$$
\begin{aligned}
&\quad\int_\Gamma\left(\frac{4(m-1)}{m-2}|\nabla\psi|^2
+R_{\gamma_i}\psi^2\right)dA\\
=&\,2\int_\Gamma\left(|\nabla\psi|^2
+\frac12R_{\gamma_i}\psi^2\right)dA
+\left(\frac{4(m-1)}{m-2}-2\right)
\int_\Gamma|\nabla\psi|^2\,dA\\
\ge&\,\delta_{i,\Gamma}\int_\Gamma\psi^2\,dA>0.
\end{aligned}
$$
The Rayleigh characterization consequently gives
$$
 \lambda_1(L_{\gamma_i}|_\Gamma)\ge\delta_{i,\Gamma}>0.
$$
Choose a positive first eigenfunction
$u$ on $\Gamma$. 
If $L_{\gamma_i}u=\lambda_1u$, then
$$
  k_i=u^{4/(m-2)}\gamma_i,
  \qquad
  R_{k_i}=u^{-(m+2)/(m-2)}L_{\gamma_i}u
             =\lambda_1u^{-4/(m-2)}>0.
$$
Taking these metrics componentwise gives the asserted metric on $\Sigma_i$.
\end{proof}

\section{Block filling}\label{sec3}

We formulate the filling criterion on a compact manifold with
boundary. First let us specify minimization using competitors that stay inside the
manifold. If $V\subset W$ is a compact smooth neighborhood of a union
$\Sigma$ of boundary components, write
\[
 V\cap\partial W=\Sigma,\qquad
 \partial V=\Sigma\sqcup S,\qquad S\subset\Int W.
\]

\begin{definition}
\label{def:one-sided-bubble}
Let $V$, $\Sigma$, and $S$ be as above, and let $h\in C^\infty(V)$ be
smooth up to the boundary. We say that $\Sigma$ is a \emph{one-sided
minimizing $\mu$-bubble for $h$ in $V$} if $H_g=h$ on $\Sigma$, with
respect to the outward normal of $W$, and
\begin{equation}\label{eq:one-sided-minimization}
  \mathcal A_{h,V}(V)\le \mathcal A_{h,V}(F)
\end{equation}
for every finite-perimeter set $F\subset V$ that agrees with $V$ near
$S$, where
\begin{equation*}
 \mathcal A_{h,V}(F)
 :=P(F;V)
        -\int_V(\chi_F-1)h\,dV_g.
\end{equation*}
\end{definition}

We can establish the following filling proposition.
\begin{proposition}
\label{prop:block-filling}
Let $(W^n,g)$ be a compact oriented Riemannian manifold with smooth
boundary $\partial W=\bigsqcup_{\alpha=1}^m\Sigma_\alpha$, $n\ge3$,
and assume $R_g>0$ on $W$. Fix smooth metrics
\[
 k_\alpha\in[g|_{\Sigma_\alpha}],\qquad R_{k_\alpha}>0.
\]
Suppose that there are pairwise disjoint
compact smooth neighborhoods $V_\alpha\subset W$, with
\[
 V_\alpha\cap\partial W=\Sigma_\alpha,\qquad
 \partial V_\alpha=\Sigma_\alpha\sqcup S_\alpha,
 \qquad S_\alpha\subset\Int W,
\]
and functions $h_\alpha\in C^\infty(V_\alpha)$ such that:
\begin{enumerate}[label=\textup{(\roman*)}]
 \item $\Sigma_\alpha$ is a one-sided minimizing $\mu$-bubble for
 $h_\alpha$ in $V_\alpha$, in the sense of
 \cref{def:one-sided-bubble};
 \item $h_\alpha\ge0$ in a neighborhood of $S_\alpha$;
 \item on $\spt_{V_\alpha}(h_\alpha^-)$, where
 $h_\alpha^-:=\max\{-h_\alpha,0\}$,
 \begin{equation}\label{eq:filling-q}
 q_\alpha:=R_g+\frac{n}{n-1}h_\alpha^2
                   -2|\nabla h_\alpha|>0.
 \end{equation}
\end{enumerate}
Then there is a smooth metric $G$ on $W$ such that
\begin{equation}\label{eq:block-product-conclusion}
 R_G>0,\qquad G=dt^2+k_\alpha
       \quad\text{near every }\Sigma_\alpha.
\end{equation}
\end{proposition}

In the following,
 for a bounded divergence-measure field,
$[Z,\nu]$ denotes its normal trace; the required trace results are
recorded in Appendix~\ref{appendixB}.

\begin{lemma}\label{lem:calibration}
Let $V$, $\Sigma$, and $S$ be as above, let
$h\in C^\infty(V)$, and assume the minimizing inequality \eqref{eq:one-sided-minimization}.
Let $U\subset V$ be a compact smooth neighborhood of $\Sigma$ with
\[
 \partial U=\Sigma\sqcup T,\qquad
 T\subset V^\circ,\qquad U\cap S=\varnothing.
\]
Then there is $Z\in L^\infty(TU)$ satisfying
\[
 |Z|\le1,\qquad \diver Z\in L^\infty(U),\qquad
 \diver Z\le h\ \text{a.e. on }U,\qquad
 [Z,\nu_W]=1\ \text{a.e. on }\Sigma.
\]
\end{lemma}

\begin{proof}
The proof is divided into three steps.

\emph{Step 1: relaxation with boundary values.}
On $\partial U$ set $\xi=0$ on $\Sigma$ and $\xi=1$ on $T$, and define
\[
 \mathcal T_\xi(v)
 :=|Dv|(U^\circ)+\int_{\partial U}|Tv-\xi|\,dA_g.
\]
For $v\in BV(U^\circ)$ with $0\le v\le1$, extend $v$ by $1$ on
$V\setminus U$ and denote this extension by $\widetilde v$.
For almost every $t\in(0,1)$, the set
$F_t=\{\widetilde v>t\}$ has finite perimeter and equals $V$ near $S$.
The BV gluing and coarea formulas
\cite[Chapter~3]{AmbrosioFuscoPallara2000} give
\[
\begin{aligned}
\int_0^1 P(F_t;V)\,dt
&=|Dv|(U^\circ)+\int_T|Tv-1|\,dA_g
  +\int_\Sigma Tv\,dA_g+\operatorname{Area}_g(S)\\
&=\mathcal T_\xi(v)+\operatorname{Area}_g(S).
\end{aligned}
\]
Since $\mathcal A_{h,V}(V)=\mathcal T_\xi(1)+\operatorname{Area}_g(S)$,
the fixed artificial-boundary term cancels when we integrate
\eqref{eq:one-sided-minimization} in $t$.
The layer-cake formula then yields
\[
 \mathcal T_\xi(v)-\int_Uhv\,dV_g
 \ge \mathcal T_\xi(1)-\int_Uh\,dV_g.
\]
Thus the constant function $1$ minimizes the relaxed functional
under the constraint $0\le v\le1$.

\smallskip
\emph{Step 2: exact penalization.}
Put $\Phi(s)=\dist(s,[0,1])$ and choose
$\Lambda>\|h\|_{L^\infty(U)}$. On $L^2(U)$ consider
\begin{equation}\label{unconfunc}
 \mathcal F(v)=\mathcal T_\xi(v)-\int_Uhv\,dV_g
                       +\Lambda\int_U\Phi(v)\,dV_g.
\end{equation}
Projection $\Pi(s)=\min\{1,\max\{0,s\}\}$ does not increase the
variation or boundary terms, since $\xi$ takes values in $[0,1]$.
Moreover, $-hv+\Lambda\Phi(v)\ge-h\Pi(v)$ pointwise.
Hence $1$ is an unconstrained minimizer of \eqref{unconfunc}.
The last two terms are continuous convex functions on $L^2(U)$, so
the subdifferential sum rule gives
\begin{equation}\label{subgraeq}
 p-h+\Lambda s=0,\qquad
 p\in\partial_{L^2}\mathcal T_\xi(1),\qquad
 s(x)\in\partial\Phi(1)=[0,1]\ \text{a.e.}
\end{equation}
Consequently $p\in L^\infty(U)$ and $p\le h$.

\smallskip
\emph{Step 3: the boundary trace.}
By \cref{thm:dirichlet-tv}, there is a field $z$ such that
\[
 |z|\le1,\qquad p=-\diver z,\qquad
 [z,\nu_U]\in\sign(\xi-1)\quad\text{on }\partial U.
\]
On $\Sigma$, the last condition reads $[z,\nu_W]=-1$.
Thus $Z=-z$ has the required divergence inequality and outward trace.
\end{proof}

\begin{lemma}\label{lem:robin-filling}
Let $(W^n,g)$ be a compact connected oriented Riemannian manifold with
nonempty smooth boundary, $n\ge3$, and let $k\in[g|_{\partial W}]$ be
smooth with $R_k>0$. Put $a_n=4(n-1)/(n-2)$. Suppose
$X\in L^\infty(TW)$, $\diver X\in L^\infty(W)$, and $\delta>0$
satisfy
\begin{align}
 R_g+a_n(\diver X-|X|^2)&\ge\delta
       &&\text{a.e. on }W,\label{eq:filling-potential}\\
 2H_g-a_n[X,\nu]&\ge0
       &&\text{a.e. on }\partial W.\label{eq:filling-boundary}
\end{align}
Then there is a smooth metric $G$ on $W$ with $R_G>0$ and
$G=dt^2+k$ near $\partial W$.
\end{lemma}

\begin{proof}
The conformal Robin quadratic form is
\[
 E_g(u)=a_n\int_W|\nabla u|^2\,dV_g
       +\int_W R_gu^2\,dV_g
       +2\int_{\partial W}H_gu^2\,dA_g.
\]
By the square identity in \cref{lem:square} and the hypotheses,
\[
 E_g(u)\ge a_n\|\nabla u+Xu\|_2^2+\delta\|u\|_2^2
 \qquad\text{for every }u\in H^1(W).
\]
Writing $C=\|X\|_\infty$, we have
\[
 \|u\|_{H^1(W)}^2
 \le 2\|\nabla u+Xu\|_2^2+(2C^2+1)\|u\|_2^2.
\]
Consequently $E_g(u)\ge c\|u\|_{H^1(W)}^2$, where
\[
 c=\min\left\{\frac{a_n}{2},\frac{\delta}{2C^2+1}\right\}>0.
\]
The critical Sobolev inequality gives
\begin{equation}\label{eq:filling-positive-quotient}
 \inf_{0\ne u\in H^1(W)}
 \frac{E_g(u)}
 {\left(\int_W|u|^{2n/(n-2)}\,dV_g\right)^{(n-2)/n}}>0.
\end{equation}
Since $k\in[g|_{\partial W}]$, this verifies the positivity condition
for the prescribed boundary conformal class in the
Akutagawa--Botvinnik extension theorem \cref{thm:ab-filling}, which gives the required extension
with an exact product near the boundary.
\end{proof}

\begin{proof}[Proof of \cref{prop:block-filling}]
Put $a_n=4(n-1)/(n-2)$ and $b_n=(n-2)/(2(n-1))$.
Fix an index $\alpha$. Since $h_\alpha\ge0$ near the
artificial boundary $S_\alpha$, we can trim a small collar of
$S_\alpha$ to obtain a compact smooth neighborhood
$U_\alpha\subset V_\alpha$ of $\Sigma_\alpha$ such that
\[
 \partial U_\alpha=\Sigma_\alpha\sqcup T_\alpha,
 \qquad U_\alpha\cap S_\alpha=\varnothing,
\]
and $U_\alpha$ contains the support in $V_\alpha$ of
$h_\alpha^-:=\max\{-h_\alpha,0\}$ in its relative interior in $W$.
In particular, $h_\alpha^-$ vanishes near $T_\alpha$.
By \cref{lem:calibration}, there is a field $Z_\alpha$ on $U_\alpha$
with $|Z_\alpha|\le1$, $\diver Z_\alpha\le h_\alpha$, and
$[Z_\alpha,\nu_W]=1$ on $\Sigma_\alpha$.

Define $X_\alpha=-b_nh_\alpha^-Z_\alpha$ on $U_\alpha$ and extend it
by zero to $W$. There is no jump at the artificial boundary because
$h_\alpha^-$ vanishes near $T_\alpha$. Thus
\[
 X:=\sum_{\alpha=1}^m X_\alpha\in L^\infty(TW),
 \qquad \diver X\in L^\infty(W).
\]
The supports of the summands are disjoint. On $\{h_\alpha<0\}$,
we have $X_\alpha=b_nh_\alpha Z_\alpha$, and therefore
\begin{align*}
 \diver X_\alpha
 &=b_n\bigl(\langle\nabla h_\alpha,Z_\alpha\rangle
                    +h_\alpha\diver Z_\alpha\bigr)\\
 &\ge b_n\bigl(-|\nabla h_\alpha|+h_\alpha^2\bigr),
 \qquad |X_\alpha|^2\le b_n^2h_\alpha^2.
\end{align*}
Since $a_nb_n=2$ and $a_n(b_n-b_n^2)=n/(n-1)$,
\begin{equation}\label{eq:block-field-potential}
 R_g+a_n(\diver X_\alpha-|X_\alpha|^2)
 \ge R_g+\frac{n}{n-1}h_\alpha^2-2|\nabla h_\alpha|
 =q_\alpha.
\end{equation}
Elsewhere $X_\alpha$ and its divergence vanish almost everywhere.
Each support $\spt_{V_\alpha}(h_\alpha^-)$ is compact, and
$q_\alpha$ is strictly positive there. Taking the minimum of
$\min_W R_g$ and the positive minima of $q_\alpha$ over the
nonempty supports gives \eqref{eq:filling-potential} for some
$\delta>0$. If all supports are empty, take $\delta=\min_W R_g$.

On each $\Sigma_\alpha$, the equality
$H_g=h_\alpha$ and the trace condition give
\begin{equation}\label{eq:block-field-boundary}
 2H_g-a_n[X,\nu_W]
 =2h_\alpha+2h_\alpha^-=2h_\alpha^+\ge0.
\end{equation}
Apply \cref{lem:robin-filling} on each connected component with
nonempty boundary, using $k=\bigsqcup_\alpha k_\alpha$. On components
without boundary keep $g$. This proves the proposition.
\end{proof}

\begin{corollary}
\label{cor:intrinsic-all-bubbles}
Assume the geometric and minimizing hypotheses of
\cref{prop:block-filling}, with $n\ge4$, without prescribing the
metrics $k_\alpha$. If, in addition, $q_\alpha>0$ on every
$\Sigma_\alpha$, then each induced boundary
conformal class contains a positive-scalar-curvature metric, and any
choice of such metrics $k_\alpha$ extends to a metric $G$ satisfying
\eqref{eq:block-product-conclusion}.
\end{corollary}

\begin{proof}
Fix a connected component $\Gamma$ of $\Sigma_\alpha$ and put
$\gamma=g|_\Gamma$. For a nonnegative smooth function $\psi$, move
$\Gamma$ inward with normal velocity $-\psi\nu_W$, keeping the
region fixed away from $\Gamma$. These are admissible competitors
in \eqref{eq:one-sided-minimization}. The first derivative of the
functional is $\int_\Gamma(h_\alpha-H_g)\psi\,dA_\gamma=0$.
The one-sided minimum therefore gives a nonnegative second derivative:
\[
 Q(\psi):=\int_\Gamma\left(|\nabla_\Gamma\psi|^2
 -(\operatorname{Ric}_g(\nu_W,\nu_W)+|\mathrm{II}|^2
                  +\nu_W(h_\alpha))\psi^2\right)dA_\gamma\ge0.
\]
Approximating $|\psi|$ by nonnegative smooth functions extends this
inequality to arbitrary smooth $\psi$, since $Q(|\psi|)=Q(\psi)$.
The traced Gauss equation, $H_g=h_\alpha$, and
$|\mathrm{II}|^2\ge h_\alpha^2/(n-1)$ now imply
\begin{equation}\label{eq:block-boundary-stability}
 \int_\Gamma\left(2|\nabla_\Gamma\psi|^2
                    +R_\gamma\psi^2\right)dA_\gamma
 \ge\int_\Gamma q_\alpha\psi^2\,dA_\gamma.
\end{equation}
With $d=n-1\ge3$, the coefficient $4(d-1)/(d-2)$ is greater than
$2$. The conformal Laplacian of $(\Gamma,\gamma)$ consequently has
positive first eigenvalue. A positive first eigenfunction gives a
positive-scalar-curvature metric conformal to $\gamma$.
Apply \cref{prop:block-filling} with any chosen metrics in these
positive boundary conformal classes.
\end{proof}

We now apply the intrinsic criterion to the exhaustion from
\cref{prop:block-decomposition}. Put
\[
 W_0=K_1,\qquad W_i=K_{i+1}\setminus\Int K_i\quad(i\ge1).
\]

\begin{corollary}
\label{prop:block-metrics}
Fix the metrics $k_i\in[g|_{\Sigma_i}]$ with $R_{k_i}>0$ from
\cref{prop:interface-psc}. There are smooth metrics $G_i$ with
$R_{G_i}>0$ on $W_i$ such that
\[
 G_0=dt^2+k_1\quad\text{near }\partial W_0,
\]
and, for $i\ge1$,
\[
 G_i=dt^2+k_i\quad\text{near }\Sigma_i,\qquad
 G_i=dt^2+k_{i+1}\quad\text{near }\Sigma_{i+1}.
\]
\end{corollary}

\begin{proof}
We check the local hypotheses inside each block. The prescribed
function must be oriented using the outward normal of that block:
it is $h_1$ on the boundary of $W_0$, while for $W_i$, $i\ge1$, it is
\[
 f_i=-h_i\quad\text{at }\Sigma_i,\qquad
 f_{i+1}=h_{i+1}\quad\text{at }\Sigma_{i+1}.
\]
The construction in the proof of \cref{prop:block-decomposition}
gives $K_j\cap A_j=\Omega_j$ up to volume-null sets, and the ordered
bands satisfy $b_j<a_{j+1}$. Thus $W_i$ occupies the complementary
side of $\Omega_i$ in $A_i$ and the $\Omega_{i+1}$-side in $A_{i+1}$.
The prescribed-function signs follow from
Item (\ref{item:band-normal-orientation}). Complementing a minimizing region
reverses the sign of the prescribed function and preserves the
minimizing inequality.

For each interface, truncate the appropriate band on the side
occupied by $W_i$ to obtain a compact smooth neighborhood $V$ of
that interface in $W_i$. The truncation can be made in the fixed
end collar: the pole limits in Item (\ref{item:band-collars}) ensure
$f\ge0$ near its artificial boundary. More explicitly, on the
$\Omega_j$-side we truncate near $\partial_{\rm in}A_j$, where
$f=h_j>0$; on the complementary side we truncate near
$\partial_{\rm out}A_j$, where $f=-h_j>0$. The resulting
neighborhoods are disjoint because the original bands are disjoint.

Let $E$ be the block-side minimizing region in $A_j$. Given an
admissible $F\subset V$, set $E_F=(E\setminus V)\cup F$, leaving
the region unchanged elsewhere in the band. This is an admissible
competitor in \cref{def:minimizing-mu-bubble}. The BV gluing formula gives
\[
 P(E_F;A_j^\circ)-P(E;A_j^\circ)=P(F;V)-P(V;V).
\]
Indeed, the exterior trace across $\Sigma_j$ is zero, so the
interface contribution is $\int_{\Sigma_j}T\chi_F\,dA_g$; the
artificial-boundary terms $\operatorname{Area}_g(S)$ on the right
cancel. The volume difference is
\[
 \int_{A_j^\circ}(\chi_{E_F}-\chi_E)f\,dV_g
 =\int_V(\chi_F-1)f\,dV_g.
\]
Hence minimization in the band gives
$\mathcal A_{f,V}(V)\le\mathcal A_{f,V}(F)$.
The original first variation also gives $H_g=f$ in this orientation.

Finally, $R_g>0$ on each compact block, and \eqref{eq:band-potential}
implies \eqref{eq:filling-q}, since $f^2=h_j^2$ and
$|\nabla f|=|\nabla h_j|$. Applying \cref{prop:block-filling} with
the prescribed metrics $k_j$ gives the required metrics $G_i$,
with the same $k_j$ on the two blocks adjacent to $\Sigma_j$.
\end{proof}

\section{Global construction}\label{sec4}

To obtain a uniform lower bound of scalar curvature, the block metrics must be
rescaled independently.  This generally induces distinct constant
rescalings of $k_i$ on the two sides of $\Sigma_i$.  The following proposition
performs the rescaling and interpolation and establishes completeness.

\begin{proposition}
\label{prop:uniform-gluing}
Let $n\ge 3$, and $M^n$ be a connected noncompact manifold admitting an
exhaustion by compact smooth codimension-zero submanifolds
$$
  K_i\Subset\operatorname{Int}K_{i+1},
  \qquad M=\bigcup_iK_i,
$$
and put
$$
  \Sigma_i=\partial K_i,
  \qquad W_0=K_1,
  \qquad W_i=K_{i+1}\setminus\operatorname{Int}K_i.
$$
Suppose each $\Sigma_i$ carries a positive-scalar-curvature metric $k_i$, and each $W_i$
carries a positive-scalar-curvature metric $G_i$ that is exactly $dt^2+k_j$ near every
boundary component $\Sigma_j$.  Then $M$ carries a smooth complete metric $\bar g$ satisfying
$$
  R_{\bar g}\ge1.
$$
\end{proposition}

\begin{proof}
The proof has three steps: rescale the compact blocks, join the two scales
at each interface by a sufficiently long warped product, and use the neck
lengths to prove completeness.

\smallskip
\noindent\emph{Step 1: rescale the blocks.}
Set
$$
  \sigma_i=\min_{W_i}R_{G_i}>0,
  \qquad
  \kappa_j=\min_{\Sigma_j}R_{k_j}>0.
$$
Choose $\lambda_0>0$ and $\lambda_i>0$, $i\ge1$, such that
$$
  \lambda_0^{-2}\min\{\sigma_0,\kappa_1\}\ge4,
$$
and
$$
  \lambda_i^{-2}
  \min\{\sigma_i,\kappa_i,\kappa_{i+1}\}\ge4
  \qquad(i\ge1).
$$
Then the metric $\widehat G_i=\lambda_i^2G_i$ has
scalar curvature $R_{\widehat G_i}=\lambda_i^{-2}R_{G_i}\ge4$.
Near $\Sigma_j\subset\partial W_i$, after the collar-coordinate change
$\tau=\lambda_it$, it is
$$
  \widehat G_i=d\tau^2+\lambda_i^2k_j.
$$

\smallskip
\noindent\emph{Step 2: interpolate the boundary scales.}
At $\Sigma_i$, insert $N_i=[0,L_i]\times\Sigma_i$ with
$$
  g_i^{\rm neck}=ds^2+w_i(s)^2k_i,
$$
where $w_i$ is constant near the endpoints and
$$
  \log w_i(s)
  =(1-\chi(s/L_i))\log\lambda_{i-1}
   +\chi(s/L_i)\log\lambda_i.
$$
Here $\chi\colon[0,1]\to[0,1]$ is a fixed and smooth function, with
$\chi\equiv0$ near $0$ and $\chi\equiv1$ near $1$.  Put
$$
  m=n-1,
  \qquad
  \theta_i=\log\frac{\lambda_i}{\lambda_{i-1}},
  \qquad
  C_1=\|\chi'\|_\infty,
  \qquad
  C_2=\|\chi''\|_\infty.
$$
Then
$$
  \frac{w_i'}{w_i}=\frac{\theta_i}{L_i}\chi'(s/L_i),
$$
and
$$
  \frac{w_i''}{w_i}
  =\frac{\theta_i}{L_i^2}\chi''(s/L_i)
   +\frac{\theta_i^2}{L_i^2}\chi'(s/L_i)^2.
$$
The scalar curvature of the warped product is
$$
  R_{g_i^{\rm  neck}}
  =w_i^{-2}R_{k_i}
   -2m\frac{w_i''}{w_i}
   -m(m-1)\left(\frac{w_i'}{w_i}\right)^2.
$$
The function $w_i$ lies between its endpoint values.  The two scale
conditions therefore imply $w_i^{-2}R_{k_i}\ge4$, and hence
$$
  R_{g_i^{\rm neck}}
  \ge4-
  \frac{2mC_2|\theta_i|+m(m+1)C_1^2\theta_i^2}{L_i^2}.
$$
Choose
$$
  L_i\ge
  \max\left\{1,
  \sqrt{\frac{2mC_2|\theta_i|+m(m+1)C_1^2\theta_i^2}{3}}
  \right\}.
$$
Then $R_{g_i^{\rm neck}}\ge1$.

Trim pairwise disjoint smaller product bicollars from the adjacent scaled
blocks and insert $N_i$ between the exposed product boundaries.  Since $w_i$
is constant near both endpoints, the metrics agree on open product collars
and glue smoothly.  Replacing a bicollar by a longer finite cylinder does
not change the diffeomorphism type.  The interfaces leave every compact
subset of $M$, so these replacements are locally finite.  They produce a
smooth metric on a manifold diffeomorphic to $M$; pulling it back gives a
smooth metric $\bar g$ on $M$.  The notation $K_i$, $\Sigma_i$, and $N_i$ is
retained for the corresponding images.  The block and neck
estimates give $R_{\bar g}\ge1$.

\smallskip
\noindent\emph{Step 3: completeness.}
Let $\gamma\colon[0,T)\to M$ be a locally Lipschitz curve that leaves every
compact subset.  Choose $N$ with $\gamma(0)\in\operatorname{Int}K_N$.
The two faces of each $N_i$ separate the already glued inner blocks from the
remaining outer blocks.  Hence, by continuity, a curve that eventually
leaves the compact inner region contains a subarc joining those two faces.
For every sufficiently large $i$, a subarc of $\gamma$ crosses one connected
component of $N_i$ from $s=0$ to $s=L_i$.  If $J_i\subset[0,T)$ is its
parameter interval, then
$$
  L_{\bar g}(\gamma|_{J_i})
  \ge L_i\ge1.
$$
These crossing subarcs have disjoint parameter intervals, so the lower
bounds sum to infinity.  Every curve leaving all compact subsets therefore
has infinite length, and the escaping curve formulation of Hopf--Rinow proves completeness.
\end{proof}
\begin{proof}[Proof of Theorem~\ref{thm:intro-main}]
Fix a compact smooth domain $K_0\Subset M$.
\Cref{prop:block-decomposition} gives an exhaustion $K_i$ whose
interfaces $\Sigma_i$ are smooth minimizing $\mu$-bubbles for the
prescribed functions $h_i$. The upper dimension bound is used here
for smooth regularity. By \cref{prop:interface-psc}, choose once and
for all metrics $k_i\in[g|_{\Sigma_i}]$ with $R_{k_i}>0$; the lower
bound $n\ge4$ makes these interfaces at least three-dimensional.
The band signs and curvature inequalities verify the hypotheses of
\cref{prop:block-filling}. Its application in
\cref{prop:block-metrics} gives metrics $G_i$ on the compact blocks
with $R_{G_i}>0$ and the same product metric $dt^2+k_i$ on the two
blocks adjacent to $\Sigma_i$. These data satisfy
\cref{prop:uniform-gluing}, which gives a smooth complete metric
$\bar g$ on $M$ with $R_{\bar g}\ge1$.
\end{proof}

\appendix
\section{Akutagawa--Botvinnik's extension theorem and the conformal Robin form}
\label{appendixA}

This appendix recalls the part of Akutagawa and Botvinnik's relative
Yamabe theory \cite{AkutagawaBotvinnik2002a,AkutagawaBotvinnik2002b}
that underlies the filling argument in \cref{sec3}. We state their
extension theorem in the form used here and record a revised
eigenfunction estimate communicated to us by the authors. We also
explain the relation to the conformal Robin quadratic form and record
the square identity used in our application.

Let $W^n$, $n\ge3$, be a compact smooth manifold with boundary. We use
the outward unit normal and take $H_g$ to be the trace of the second
fundamental form. Define
\[
 E_g(u)=\frac{4(n-1)}{n-2}\int_W|\nabla u|^2\,dV_g
       +\int_W R_gu^2\,dV_g
       +2\int_{\partial W}H_gu^2\,dA_g,
\]
and write the relative Yamabe constant of $[g]$ in the equivalent
form
\begin{equation}\label{eq:relative-yamabe-definition}
 Y_{[g]}(W,\partial W;[g|_{\partial W}])
 :=\inf_{0\ne u\in H^1(W)}
 \frac{E_g(u)}
 {\left(\int_W|u|^{2n/(n-2)}\,dV_g\right)^{(n-2)/n}}.
\end{equation}
For a conformal class $[k]$ on $\partial W$, the relative Yamabe
invariant of Akutagawa and Botvinnik is
\[
 Y(W,\partial W;[k])
 :=\sup_{\substack{\overline{\mathcal C}\text{ a conformal class on }W\\
                   \overline{\mathcal C}|_{\partial W}=[k]}}
       Y_{\overline{\mathcal C}}(W,\partial W;[k]).
\]
Here $Y_{\overline{\mathcal C}}$ is the fixed-class constant for any
representative of $\overline{\mathcal C}$. In particular,
\eqref{eq:filling-positive-quotient} implies
$Y(W,\partial W;[g|_{\partial W}])>0$.

The conformal Laplacian and conformal Robin operator are
\begin{equation}\label{conformalrobin}
 L_g=-\frac{4(n-1)}{n-2}\Delta_g+R_g,
 \qquad B_g=\partial_\nu+\frac{n-2}{2(n-1)}H_g,
\end{equation}
where $\Delta=\diver\nabla$.
For a positive smooth function $\phi$, set
$\widehat g=\phi^{4/(n-2)}g$. The conformal covariance of $L_g$ and $B_g$
is
\begin{equation}\label{eq:robin-conformal-covariance}
  L_{\widehat g}v
  =\phi^{-(n+2)/(n-2)}L_g(\phi v),
  \qquad
  B_{\widehat g}v
  =\phi^{-n/(n-2)}B_g(\phi v).
\end{equation}
Consequently,
\begin{equation}\label{eq:yamabe-quotient-covariance}
  E_{\widehat g}(v)=E_g(\phi v),
  \qquad
  \int_W|v|^{\frac{2n}{n-2}}\,dV_{\widehat g}
  =\int_W|\phi v|^{\frac{2n}{n-2}}\,dV_g.
\end{equation}
To compare \eqref{eq:relative-yamabe-definition} with the definition
in \cite{AkutagawaBotvinnik2002a,AkutagawaBotvinnik2002b}, choose a
smooth positive function $\phi$ with $B_g\phi=0$ on $\partial W$.
Such a function can be prescribed in a boundary collar and extended
positively to $W$. Then $H_{\widehat g}=0$ by
\eqref{eq:robin-conformal-covariance}, and
\eqref{eq:yamabe-quotient-covariance} preserves the quotient.
Smooth functions satisfying the homogeneous Neumann condition are
dense in $H^1(W)$, by even reflection in a boundary collar and
smoothing. For any such function $v$, the functions
$\sqrt{v^2+\varepsilon^2}$ are smooth, positive, and satisfy the same
boundary condition; they converge to $|v|$ in $H^1(W)$ as
$\varepsilon\downarrow0$. Thus the infimum in
\eqref{eq:relative-yamabe-definition} agrees with the infimum over
positive smooth conformal factors preserving the minimal boundary
condition, as in Akutagawa and Botvinnik's formulation.

The following is the implication of
\cite[Corollary~B]{AkutagawaBotvinnik2002a} used in
\cref{lem:robin-filling}.

\begin{theorem}[Akutagawa--Botvinnik]
\label{thm:ab-filling}
Let $W^n$ be a compact smooth oriented manifold with nonempty boundary
and dimension $n\ge3$, and let $k$ be a smooth
positive-scalar-curvature metric on $\partial W$. If
\[
  Y(W,\partial W;[k])>0,
\]
then $k$ extends to a positive-scalar-curvature metric $G$ on $W$
such that
\[
  G=dt^2+k
\]
near $\partial W$.
\end{theorem}

Akutagawa and Botvinnik obtain this result from their extension
theorem \cite[Theorem~3]{AkutagawaBotvinnik2002a}. The eigenfunction
estimate in Lemma~6 of that paper requires the revision described
below. We are grateful to Akutagawa and Botvinnik for communicating
the revised statement and proof and for permitting us to include them
here. The revision is due to them; we present it in notation compatible
with this paper.

We follow the cylindrical construction of Akutagawa and Botvinnik
\cite[Sections~3.2 and~4, pp.~825--829]{AkutagawaBotvinnik2002a}.
For this construction, work on a connected component of $W$ with
nonempty boundary and set $Z=\partial W$, including all its boundary
components. Thus $W$ is connected, whereas $Z$ need not be. Write
\[
X=W\cup_Z\bigl(Z\times[0,\infty)\bigr),\qquad
X(\ell)=W\cup_Z\bigl(Z\times[0,\ell]\bigr).
\]
As in
\cite[eq.~(6), p.~826]{AkutagawaBotvinnik2002a},
\begin{equation}\label{eq:ab6-core-spectrum}
\nu_1(\bar g)=
\inf_{f\in C^\infty(X(1)),\,f\not\equiv0}
\frac{\int_{X(1)}
       \bigl(\frac{4(n-1)}{n-2}|df|_{\bar g}^2+R_{\bar g}f^2\bigr)
       \,d\sigma_{\bar g}}
     {\int_{X(1)}f^2\,d\sigma_{\bar g}}.
\end{equation}
Equivalently, this is the infimum over nonzero $H^1(X(1))$
functions. In particular, $\nu_1(\bar g)$ is the first Neumann
eigenvalue on the compact core, whereas the eigenvalues
$\lambda_\ell$ below are Dirichlet eigenvalues.

Fix the background metric supplied by
\cite[Lemmas~1--2]{AkutagawaBotvinnik2002a}, so that
\begin{equation}\label{eq:ab6-background}
\bar g=h+dt^2\quad\text{on }Z\times[1-\varepsilon,\infty),
\qquad R_h\equiv1,\qquad \nu_1(\bar g)\ge1.
\end{equation}
Put
\[
L_{\bar g}=-\frac{4(n-1)}{n-2}\Delta_{\bar g}+R_{\bar g},
\]
and retain the volume-element notation $d\sigma_{\bar g}$.

Let $\lambda_\ell$ be the first Dirichlet eigenvalue of $L_{\bar g}$ on
$X(\ell)$, as on p.~828 of \cite{AkutagawaBotvinnik2002a}. For $\ell\ge5$,
choose its positive eigenfunction $u_\ell$ with
\begin{equation}\label{eq:ab6-eigenfunction}
\begin{gathered}
L_{\bar g}u_\ell=\lambda_\ell u_\ell
\quad\text{on }X(\ell),\qquad
u_\ell>0\quad\text{in }{\rm int}X(\ell),\\
u_\ell=0\quad\text{on }Z\times\{\ell\},\qquad
\min_{Z\times\{\ell/2\}}u_\ell=1.
\end{gathered}
\end{equation}
In this revised formulation, the eigenfunction is normalized at
$t=\ell/2$, and the uniform upper bound is asserted on the terminal
half-cylinder $Z\times[\ell/2,\ell]$.

\begin{lemma}[Akutagawa--Botvinnik, revised form of Lemma~6 in \cite{AkutagawaBotvinnik2002a}]
\label{lem:ab6-corrected}
Under \eqref{eq:ab6-background}--\eqref{eq:ab6-eigenfunction}, there is a
constant $\widehat K>0$, independent of $\ell\ge5$, such that
\begin{equation}\label{eq:ab6-conclusion}
u_\ell\le\widehat K\qquad\text{on }Z\times[\ell/2,\ell].
\end{equation}
\end{lemma}

\begin{proof}
We first record the eigenvalue estimate needed for the comparison. The
energy decomposition between $X(1)$ and $Z\times[1,\ell]$, together with
\eqref{eq:ab6-background}, gives $\lambda_\ell\ge1$; compare
\cite[Lemma~3]{AkutagawaBotvinnik2002a}. On the other hand, the function
\[
f_\ell(z,t)=\sin\frac{\pi(t-1)}{\ell-1},\qquad 1\le t\le\ell,
\]
extended by zero to $X(1)$, belongs to $H_0^1(X(\ell))$. Its Rayleigh
quotient gives
\begin{equation}\label{eq:ab6-spectral}
1\le\lambda_\ell\le
1+\frac{4(n-1)}{n-2}\left(\frac{\pi}{\ell-1}\right)^2.
\end{equation}
Consequently, on the product cylinder,
\begin{equation}\label{eq:ab6-cylinder-equation}
\Delta_{\bar g}u_\ell=-\kappa_\ell^2u_\ell,
\qquad
\kappa_\ell^2:=\frac{\lambda_\ell-1}{\frac{4(n-1)}{n-2}},
\qquad
0\le\kappa_\ell^2\le\left(\frac{\pi}{\ell-1}\right)^2.
\end{equation}

\emph{The midpoint estimate.}
We claim that there is a constant $\bar K$, independent of $\ell\ge5$, with
\begin{equation}\label{eq:ab6-midpoint}
\max_{Z\times\{\ell/2\}}u_\ell\le\bar K.
\end{equation}
If $Z$ is connected, an interior Harnack chain in
$Z\times(\ell/2-1,\ell/2+1)$ compares the maximum and minimum on
$Z\times\{\ell/2\}$. For $\ell\ge5$, these neighborhoods lie in the
product cylinder and are isometric by translation; the zeroth-order
coefficient is uniformly bounded by \eqref{eq:ab6-spectral}.
The Harnack constant is therefore independent of $\ell$, and the
normalization gives \eqref{eq:ab6-midpoint}.

To include disconnected boundaries, write
$Z=\bigsqcup_{j=1}^m Z_j$, with each $Z_j$ connected, and define
\[
m_{\ell,j}(t)=\frac{1}{\operatorname{Vol}_h(Z_j)}
              \int_{Z_j}u_\ell(z,t)\,d\sigma_h.
\]
Integrating \eqref{eq:ab6-cylinder-equation} over $Z_j$ yields
\[
m_{\ell,j}''+\kappa_\ell^2m_{\ell,j}=0,\qquad
m_{\ell,j}(\ell)=0,\qquad
m_{\ell,j}(t)>0\quad(1\le t<\ell).
\]
Solutions of this ordinary differential equation vanishing at $t=\ell$
form a one-dimensional space. Hence all the positive functions
$m_{\ell,j}$ are proportional, and
\begin{equation}\label{eq:ab6-mean-ratios}
\frac{m_{\ell,i}(\ell/2)}{m_{\ell,j}(\ell/2)}
=
\frac{m_{\ell,i}(1)}{m_{\ell,j}(1)}.
\end{equation}
A fixed Harnack chain through the connected compact core, contained
in $X(2)\Subset{\rm int}X(\ell)$, bounds all ratios on the right by
a constant $C_1$ independent of $\ell$. Here the geometry is fixed
and $\lambda_\ell$ is uniformly bounded by
\eqref{eq:ab6-spectral}. On each connected midpoint slice, the
Harnack inequality gives
\[
 \sup_{Z_j\times\{\ell/2\}}u_\ell
 \le C_2\inf_{Z_j\times\{\ell/2\}}u_\ell,
\]
where $C_2$ is independent of $\ell$ and may be chosen uniformly
over the finitely many components. The component containing the
global minimum $1$ has mean at most $C_2$. By
\eqref{eq:ab6-mean-ratios}, every component mean is at most $C_1C_2$,
and hence every component supremum is at most $C_1C_2^2$.
This proves \eqref{eq:ab6-midpoint}.

\emph{The comparison on the half-cylinder.}
Set
\begin{equation}\label{eq:ab6-comparison}
v_\ell(z,t)=
\cos\left(\frac{\pi}{\ell-1}
             \left(t-\frac{\ell+2}{2}\right)\right),
\qquad \ell/2\le t\le\ell.
\end{equation}
For $\ell\ge5$, this function satisfies
\begin{gather}
0<v_\ell\le1\quad\text{on }Z\times[\ell/2,\ell],
\qquad
\Delta_{\bar g}v_\ell=-\left(\frac{\pi}{\ell-1}\right)^2v_\ell,
\label{eq:ab6-comparison-equation}\\
v_\ell(z,\ell/2)=\cos\frac{\pi}{\ell-1}\ge\frac1{\sqrt2},
\qquad
v_\ell(z,\ell)=\sin\frac{\pi}{2(\ell-1)}>0.
\label{eq:ab6-comparison-endpoints}
\end{gather}
Define $w_\ell=u_\ell/v_\ell$. By
\eqref{eq:ab6-cylinder-equation} and
\eqref{eq:ab6-comparison-equation},
\begin{equation}\label{eq:ab6-quotient}
\begin{aligned}
\Delta_{\bar g}w_\ell
+2\left\langle\frac{\nabla v_\ell}{v_\ell},
                    \nabla w_\ell\right\rangle_{\bar g}
&=\frac{v_\ell\Delta_{\bar g}u_\ell
       -u_\ell\Delta_{\bar g}v_\ell}{v_\ell^2}\\
&=\left[\left(\frac{\pi}{\ell-1}\right)^2-\kappa_\ell^2\right]w_\ell
\ge0.
\end{aligned}
\end{equation}
For each fixed $\ell$, the coefficients of this drift operator are smooth
on the closed half-cylinder, since $v_\ell>0$ there. The maximum principle
and $w_\ell=0$ on $Z\times\{\ell\}$ imply
\[
\sup_{Z\times[\ell/2,\ell]}w_\ell
\le
\max\left\{
\sup_{Z\times\{\ell/2\}}\frac{u_\ell}{v_\ell},\,0
\right\}
\le\sqrt{2}\bar K.
\]
Since $v_\ell\le1$, we obtain
$u_\ell=v_\ell w_\ell\le\sqrt{2}\bar K$ on the half-cylinder.
Taking $\widehat K=\sqrt2\,\bar K$ proves the lemma.
\end{proof}

The preceding proof establishes the revised eigenfunction estimate.
Its use in Akutagawa and Botvinnik's extension argument also requires
corresponding adjustments to the subsequent cutoff construction in
\cite[Section~4]{AkutagawaBotvinnik2002a}. We do not reproduce those
steps here; the extension result invoked in this paper is their
theorem, stated as \cref{thm:ab-filling} above.

We conclude with the square identity used to verify positivity of
the conformal Robin form in \cref{sec3}. For
$X\in L^\infty(TW)$ with $\diver X\in L^\infty(W)$, denote its
outward normal trace by $[X,\nu]$; see \cref{appendixB} for the trace
convention.

\begin{lemma}\label{lem:square}
Let $X\in L^\infty(TW)$ with $\diver X\in L^\infty(W)$.
Then for every $u\in H^1(W)$,
\begin{align}
  E_g(u)
  ={}&\frac{4(n-1)}{n-2}\int_W|\nabla u+Xu|^2\,dV_g\notag\\
  &+\int_W\bigl(R_g+\frac{4(n-1)}{n-2}
    (\diver X-|X|^2)\bigr)u^2\,dV_g\notag\\
  &+\int_{\partial W}\bigl(2H_g-\frac{4(n-1)}{n-2}
    [X,\nu]\bigr)u^2\,dA_g.
  \label{eq:square}
\end{align}
\end{lemma}

\begin{proof}
For smooth $u$, the Gauss--Green formula gives
\[
 \int_W\langle X,\nabla(u^2)\rangle\,dV_g
 =-\int_W(\diver X)u^2\,dV_g
  +\int_{\partial W}[X,\nu]u^2\,dA_g.
\]
Expanding $|\nabla u+Xu|^2$ and substituting this identity proves
\eqref{eq:square} for smooth $u$. For general $u\in H^1(W)$, choose
smooth $u_j\to u$ in $H^1(W)$. Then $u_j^2\to u^2$ in
$W^{1,1}(W)$ and $(Tu_j)^2\to(Tu)^2$ in $L^1(\partial W)$,
where $T$ is the boundary trace. Since $X$, $\diver X$, and
$[X,\nu]$ are bounded, every term in \eqref{eq:square} passes to
the limit.
\end{proof}

\section{Total variation subgradients and normal traces}\label{appendixB}
For a bounded divergence-measure vector field $z$ there are, in general, two one-sided traces on an oriented rectifiable hypersurface. In the situation used below, the divergence is absolutely continuous and the two traces agree; $[z,\nu_E]$ denotes their common value with respect to the measure-theoretic unit outward normal $\nu_E$.

For a smooth compact Riemannian domain $U$ and
$f\in L^1(\partial U)$, set
$$
\mathcal T_f(v)=
\begin{cases}
|Dv|(U)+\displaystyle\int_{\partial U}|Tv-f|\,dA,
&v\in BV(U)\cap L^2(U),\\
+\infty,&\text{otherwise}.
\end{cases}
$$
Here $T$ is the $BV$ boundary interior trace.  The functional $\mathcal T_f$ is
proper, convex, and lower semicontinuous on $L^2(U)$.  We use the
multivalued sign
$$
 \sign(t)=
 \begin{cases}
   \{1\},&t>0,\\
   [-1,1],&t=0,\\
   \{-1\},&t<0.
 \end{cases}
$$

\begin{theorem}[G\'orny--Maz\'on]
\label{thm:dirichlet-tv}
For $u\in BV(U)\cap L^2(U)$,
$p\in\partial_{L^2}\mathcal T_f(u)$ if and only if there is
$z\in L^\infty(TU)$ such that
\begin{align*}
 |z|\le1,
 \qquad p=-\diver z\in L^2(U),\\
 \qquad (z,Du)=|Du|\qquad \text{ as  Radon measures on } U,\\
 [z,\nu_U]\in\sign(f-Tu)\qquad \text{ a.e. on }\partial U.
\end{align*}
Here $(z,Du)$ is the Anzellotti pairing, namely, $(z,Du)$ is a Radon measure and for any $\varphi\in C^\infty_0(U)$, $$\int_U \varphi\, d(z,Du)=-\int_Uu \langle z , \nabla\varphi\rangle\,dV-\int_Uu\varphi\diver z\,dV.$$
\end{theorem}

\begin{proof}
This is \cite[Theorem~4.3]{GM}, specialized to a smooth Riemannian domain. For the reader's convenience, we sketch the proof here. It suffices to treat one connected component of $U$ at a time.  Choose a
closed Riemannian extension $(\widehat M,\widehat g)$ for which
$\widehat M\setminus\overline U$ is connected.  The spaces
$\widehat M$, $U$, and $\widehat M\setminus\overline U$, with their
induced metrics and volume measures, are doubling and support weak
$(1,1)$-Poincar\'e inequalities.  The two smooth domains also satisfy the
density and Ahlfors boundary-regularity hypotheses of
\cite[Theorem~2.4]{GM}.  These facts follow by covering their compact
closures with finitely many smooth boundary charts.

The abstract tangent module embeds isometrically (denoted by $\iota$) in the measurable sections
of $T\widehat M$ \cite[Theorem~4.7]{LucicPasqualetto2020}.  If $X$ is a
module vector field and $\iota X$ is its image, then for every smooth
function $\phi$,
$$
  d\phi(X)=\langle\nabla\phi,\iota X\rangle.
$$
Testing the definition of weak divergence against smooth functions
therefore identifies it with Riemannian distributional divergence.
Substitution in the definition of the pairing in
\cite[Section~2.2 and Section~2.3]{GM} identifies that pairing, in each smooth chart,
with the Anzellotti pairing.  Moreover, $|D\chi_U|=dA$.

Consequently, \cite[Definition~4.2 and Theorem~4.3]{GM} characterizes
$p\in\partial_{L^2}\mathcal T_f(u)$ by a field $z$ satisfying
\begin{align*}
  |z|\le1,
  \qquad p=-\diver z\in L^2(U),\\
  (z,Du)=|Du|\qquad\text{ as Radon measures on }U,\\
 (z\mathbin\cdot\nu_U)^-\in\sign(Tu-f) \qquad\text{ a.e. on }\partial U.
\end{align*}
G\'orny--Maz\'on use the interior normal trace, whereas the outward trace used
here satisfies
$$
  [z,\nu_U]=-(z\mathbin\cdot\nu_U)^-.
$$
Since $\sign(-a)=-\sign(a)$, their boundary condition is equivalent to
$[z,\nu_U]\in\sign(f-Tu)$.  This proves the stated equivalence. For a
disconnected domain, we can apply the result to each component, and combine the resulting fields.
\end{proof}

\begin{lemma}
\label{lem:saturated-trace}
Let $z\in L^\infty(TU)$ with $\diver z\in L^1(U)$, and let
$E\subset U$ have finite perimeter. If
\[
  (z,D\chi_E)=|D\chi_E|
\]
as Radon measures in $U^\circ$, then, for
$|D\chi_E|$-almost every point of $\partial^*E\cap U^\circ$,
the two one-sided traces of $z$ agree, and their common
outward trace satisfies
\[
  [z,\nu_E]=-1.
\]
\end{lemma}

\begin{proof}
By \cite[Remark~4.1]{CrastaDeCicco}, applied in smooth charts, the two
one-sided traces agree because $\diver z\in L^1(U)$ is absolutely
continuous with respect to volume.  By
\cite[Theorem~3.3]{CrastaDeCicco}, the density of the pairing is the average
of these traces.  Moreover,
$$
  D\chi_E=-\nu_E|D\chi_E|.
$$
Thus the density of the pairing is the common trace in the direction
$-\nu_E$.  Pairing saturation makes that density equal to $1$; equivalently,
the common trace in the outward direction $\nu_E$ is $-1$.
\end{proof}

\end{document}